\documentclass{amsart}
\usepackage{proof}
\usepackage{tikz}
\usepackage[mathscr]{eucal}
\usepackage{amssymb}

\newtheorem{theorem}{Theorem}[section]
\newtheorem{lemma}[theorem]{Lemma}
\newtheorem{cor}[theorem]{Corollary}
\newtheorem{prop}[theorem]{Proposition}
\newtheorem{defn}[theorem]{Definition}

\newtheorem{question}{Question}

\newcommand{\deq}{\mathrel{\mathop:}=}
\newcommand{\alg}[1]{\mathbf{#1}}
\newcommand{\var}[1]{\mathsf{#1}}
\newcommand{\sipalg}[1]{\overline{\mathbf{#1}}}
\newcommand{\Bnalg}[1]{\overline{\mathbf{B}}_{#1}}

\def\bigFOand{\mathop{\text{\rm\Large{\&}}}}
\def\FOand{\mathop{\mkern10mu\&\mkern10mu}}

\newcommand\freefootnote[1]{%
  \let\thefootnote\relax%
  \footnotetext{#1}}%

\def\up#1{\mathchoice{\bigl[#1\bigr)}{[#1)}{}{}}
\def\dw#1{\mathchoice{\bigl(#1\bigr]}{(#1]}{}{}}

\tikzstyle{every label}=[label distance=0pt]
\tikzstyle{bdot}[1.5]=[circle,fill,draw,thick,minimum size=#1mm,inner sep=0pt]
\tikzstyle{dot}[1.5]=[circle,draw,thick,minimum size=#1mm,inner sep=0pt]
\tikzstyle{every edge}=[draw=black,thick]

\title[Quasivarieties of p-algebras]{Quasivarieties of p-algebras: some new results} 

\author[Kowalski]{Tomasz Kowalski$^{2,3,4}$}
\author[Słomczyńska]{Katarzyna Słomczyńska$^{1}$}

\address{$^{1}$ Department of Mathematics,
University of National Education Commission, Kraków}
\email{irena.korwin-slomczynska@up.edu.pl}

\address{$^{2}$ Department of Logic, Jagiellonian University}
\email{tomasz.s.kowalski@uj.edu.pl}

\address{$^{3}$ Department of Physical and Mathematical Sciences, La Trobe
  University} 
\email{t.kowalski@latrobe.edu.au}

\address{$^{4}$ School of Historical and Philosophical Inquiry,
  The University of Queensland} 
\email{t.kowalski@uq.edu.au}

\begin{document}

\maketitle 

\begin{abstract}
We investigate quasivarieties of (distributive) p-algebras. We sharpen some
previous results, give a better picture of the subquasivariety lattice,
and prove that quasivarieties generated by free p-algebras belong to a rather
small quasivariety characterised by excluding all but 3 subdirectly irreducibles.
\end{abstract}

\section{Introduction}

This article is a continuation of~\cite{KS24}, where we investigated free
distributive p-algebras. Here we prove some new results on the structure of
the lattice of quasivarieties of these algebras.
As in~\cite{KS24} we assume familiarity with the fundamentals of universal
algebra, say, for definiteness, as given in Part I of~\cite{Ber11}.
Our notation for the most part also follows~\cite{Ber11},  
with trivial variations such as a change of font. Everything over and
above that will be introduced as the need arises.

Recall that distributive p-algebras (henceforth simply \emph{p-algebras}),
are bounded distributive lattices endowed with a unary operation ${}^*$
satisfying the condition
$$
x\wedge y = 0 \iff x\leq y^*
$$
equivalent over bounded distributive lattices (see, e.g.,~\cite{Ber11}, Ch.~4,
p.~108) to the equations
\begin{enumerate}
\item $1^* = 0$,
\item $0^* = 1$,
\item $x\wedge(x\wedge y)^* = x\wedge y^*$.
\end{enumerate}

Let $\var{Pa}$ stand for the variety of all p-algebras.
Subdirectly irreducible algebras in $\var{Pa}$ and the lattice of
subvarieties of $\var{Pa}$ were 
described by Lee in~\cite{Lee70}. 

\begin{theorem}[Lee]\label{thm:si-descr}
Let $\alg{A}\in\var{Pa}$ be subdirectly irreducible. Then $A = B\uplus\{1\}$
and $B$ is the universe of a Boolean algebra $\alg{B} = (B;\wedge,\vee,\neg,0,e)$,
where $e$ is the top element of $\alg{B}$. The order on $A$ extends
the natural order on $B$ by requiring $1>x$ for all $x\in B$. The lattice
operations are extended accordingly, and pseudo-complementation is defined by
$$
x^*\deq\begin{cases}
         \neg x & \text{ if } x\in B\setminus\{0\}\\
         1 & \text{ if } x = 0\\
         0 & \text{ if } x = 1
       \end{cases}     
$$
\end{theorem}

Following~\cite{Ber11} we write $\sipalg{B}$ for the subdirectly
irreducible p-algebra whose Boolean part is $\alg{B}$. For finite algebras we
write $\Bnalg{n}$, where $\alg{B}_n$ is the finite Boolean algebra
with $n$ atoms. Thus, $\Bnalg{0}$ is the two-element Boolean
algebra, which we will also denote by $\alg{2}$ where convenient.

\begin{theorem}[Lee]\label{thm:var-descr}
The lattice of subvarieties of $\var{Pa}$ is a chain of order type $\omega+1$,
consisting of varieties
$$
\var{Pa}_{-1}\subsetneq \var{Pa}_0 \subsetneq \var{Pa}_1 \subsetneq \dots
\subsetneq \var{Pa}
$$
where $\var{Pa}_{-1}$ is the trivial variety, and
$\var{Pa}_k = V(\Bnalg{k})$ for $k\in\mathbb{N}$.
\end{theorem}
In particular $\var{Pa}_0$ is the variety $\var{BA}$ of Boolean algebras.
The variety $\var{Pa}_1$, of \emph{Stone algebras},
was studied, e.g., in~\cite{BH70}, \cite{BG71} and~\cite{Pri75}.
Identities defining $\var{Pa}_k$ were given in~\cite{Lee70}. We give different
ones below, referring the reader to~\cite{KS24} for a proof. 

\begin{lemma}\label{lem:width}
For any $m>0$, the variety $\var{Pa}_m$ is axiomatised by the single
identity
\begin{equation*}
\tag{$\mathbf{ib}_{m}$}\bigvee_{i=1}^{m+1} (x_i \wedge \bigwedge_{j\neq i} x_j^*)^* = 1. 
\end{equation*}
\end{lemma}

The lattice of quasivarieties of p-algebras was studied as well.
Wroński~\cite{Wro76} and independently Adams~\cite{Ada76} proved that it
is uncountable, and Gr\"atzer \emph{et al.}~\cite{GLQ80}
improved this result, showing that the lattice of
subquasivarieties of $\var{Pa}_3$ is uncountable.
Adams and Dziobiak~\cite{AD94}, using results of 
Dziobiak~\cite{Dzi85}, show that the lattice
$L^q(\var{Pa})$ of subquasivarieties of $\var{Pa}$ is Q-universal, that is,
for any quasivariety of algebras of finite type,
its subquasivariety lattice is a homomorphic image
of a sublattice of $L^q(\var{Pa})$.

\subsection{Unnested diagrams}
We briefly describe a tool that has been widely used in the theory
of quasivarieties, more or less as folklore. A good example of such use is
Hyndman and Nation~\cite{HN18}, the paragraphs preceding Theorem~2.1. 
We will state an explicit definition, but first recall that a
formula in some first-order signature $\tau$ is \emph{unnested}, if it is
of one of the following forms:
\begin{itemize}
\item $x = y$ for variables $x$ and $y$,
\item $c = y$ for a constant $c$ and a variable $y$,
\item $f(x_1,\dots,x_n) = y$ for variables $x_i,y$ and an $n$-ary
  function symbol $f$,
\item $R(x_1,\dots,x_n)$ for variables $x_i$ and an $n$-ary
  relation symbol $R$.
\end{itemize}

\begin{defn}\label{df:diagram}
Let $\alg{A}$ be an algebra of signature $\tau$. For any $a\in A$ let $x_a$
be a distinct variable. We put
\begin{itemize}
\item $\Delta^+(\alg{A}) = \bigl\{f(x_{a_1}\dots, x_{a_n}) = x_b:
f^{\alg{A}}(a_1,\dots,a_n) = b, (a_1,\dots, a_n)\in A^n,\\ f \text{ an  $n$-ary
  function in } \tau\bigr\}\cup
  \bigl\{x_c = c^{\alg{A}}: c\text{ a constant in } \tau\bigr\}$ 
\item $\Delta^-(\alg{A}) = \bigl\{\neg(f(x_{a_1}\dots, x_{a_n}) = x_b):
f^{\alg{A}}(a_1,\dots,a_n) \neq b, (a_1,\dots, a_n)\in~A^n,\\ f \text{ an  $n$-ary
  function in } \tau\bigr\}\cup
\bigl\{\neg(x_a = c): a\neq c^{\alg{A}}, c\text{ a constant in } \tau\bigr\}$
\item $\Delta(\alg{A}) = \Delta^+(\alg{A})\cup\Delta^-(\alg{A})$
\end{itemize}
If $\tau$ and $\alg{A}$ are both finite, then $\Delta^+(\alg{A})$ and
$\Delta^-(\alg{A})$ are both finite; in this case, we define the formulas
$D^+(\alg{A}) = \bigFOand \Delta^+(\alg{A})$,
$D^-(\alg{A}) = \bigFOand\Delta^-(\alg{A})$ and 
$D(\alg{A}) = D^+(\alg{A})\FOand D^-(\alg{A})$.
\end{defn}  

We will call $\Delta(\alg{A})$ (and $D(\alg{A})$ if it exists) the
\emph{unnested diagram of} $\alg{A}$. Similarly, we will call
$\Delta^+(\alg{A})$ (and $D^+(\alg{A})$) the \emph{positive unnested diagram of}
$\alg{A}$. 
The following characterisation can be easily proved, for instance by modifying the
proof of Lemma~1.4.2, in Hodges~\cite{Hod93}.

\begin{lemma}[Diagram Lemma]\label{prop:diag-lemma}
Let $\alg{A}$ and $\alg{B}$ be algebras in signature $\tau$.
\begin{enumerate}  
\item $\alg{B}\models \Delta^+(\alg{A})(\overline{b})$ for some (possibly infinite)
tuple $\overline{b}$ of elements of $B$ if and only if a homomorphic image of $\alg{A}$
is isomorphic to a subalgebra of $\alg{B}$.
\item $\alg{B}\models \Delta(\alg{A})(\overline{b})$ for some (possibly infinite)
tuple $\overline{b}$ of elements of $B$ if and only if $\alg{A}$
is isomorphic to a subalgebra of $\alg{B}$.
\end{enumerate}  
If $\alg{A}$ and $\tau$ are finite, then we can replace $\Delta$ above by $D$,
without loss of generality.
\end{lemma}

\begin{lemma}\label{lem:fin-algs-in-qvars}
Let $\mathcal{K}$ be a class of
algebras of a finite signature $\tau$. Let $\alg{A}$ be a finite algebra
of signature $\tau$. Then the following are equivalent:
\begin{enumerate}
\item $\alg{A}\in Q(\mathcal{K})$.
\item $\alg{A}\in ISP_{\textit{fin}}(\mathcal{K})$.  
\end{enumerate}
\end{lemma}

\begin{proof}
The direction from (2) to (1) is obvious. For the converse, let
$\alg{A}$ be a finite algebra of signature $\tau$.   
Then $\alg{A} \leq \prod_{i\in I} \alg{B}_i$, where 
$\alg{B}_i = \prod_{j\in J}\mathbf{C}_{i_j}/U$ is an ultraproduct
of algebras $\mathbf{C}_{i_j}$ belonging to
$\mathcal{K}$. Since $\alg{A}$ is
finite, finitely many coordinates suffice to separate points of $A$, and so
$\alg{A}\leq \prod_{i\in I_0} \alg{B}_i$ for some finite $I_0\subseteq I$.

Let $\mathbf{D}_i$ be the projection of $\alg{A}$ on the $i$-th coordinate,
so that $\mathbf{D}_i\leq \alg{B}_i$. Then $\mathbf{D}_i$ is finite for
each $i\in I_0$. Take the diagram $D(\mathbf{D}_i)$
(cf. Definition~\ref{df:diagram}) and
consider the sentence $\exists\overline{x}:D(\mathbf{D}_i)$ obtained from
$D(\mathbf{D}_i)$ by quantifying existentially over all the variables.
Then, by Lemma~\ref{prop:diag-lemma},
we have that an algebra $\mathbf{E}$ satisfies
$\exists\overline{x}:D(\mathbf{D}_i)$ if and only if $\mathbf{D}_i\leq \mathbf{E}$.
It follows that $\alg{B}_i\models \exists\overline{x}:D(\mathbf{D}_i)$ for
each $i\in I_0$. 
But since $\alg{B}_i = \prod_{j\in J}\mathbf{C}_{i_j}/U$, we have that
$\{j\in J: \mathbf{C}_{i_j}\models \exists\overline{x}:D(\mathbf{D}_i)\}\in U$.
Pick any $j\in J$ for which
$\mathbf{C}_{i_j}\models \exists\overline{x}:D(\mathbf{D}_i)$.
Then  $\mathbf{D}_{i}\leq \mathbf{C}_{i_j}$,
and since the algebras $\mathbf{D}_{i_j}$ are projections of $\alg{A}$ we get
that $\alg{A}\leq \prod_{i\in I_0}\mathbf{C}_{i_j}$. As
$\mathbf{C}_{i_j}\in \mathcal{K}$, we obtain that
$\alg{A}\in ISP(\mathcal{K})$.
\end{proof}  

To see that the finite signature assumption cannot be relaxed, consider the class
$\mathcal{K} = \{Q(\alg{A}_i): i\in\mathbb{N}\}$ 
where $\alg{A}_i = (\{0,1\}; \wedge, \vee, (f_j:j\in\mathbb{N}),0,1)$
where $\wedge$ and $\vee$ are lattice operations on $\{0,1\}$ with respect to
the natural ordering, and each $f_j$ is a unary operation defined by
$$
f_j(x) = \begin{cases}
           x & \text{ if } j\neq i,\\
           0 & \text{ if } j = i.
\end{cases}               
$$
Let $\mathbf{D}$ be a two-element bounded lattice with additional unary operations
$(f_j:j\in\mathbb{N})$ which all are identity maps. It is easy to see that
$\mathbf{D} \in ISP_U(\mathcal{K})$, but
$\mathbf{D} \notin ISP(\mathcal{K})$, so in particular
$\mathbf{D} \notin ISP_\textit{fin}(\mathcal{K})$. 

\subsection{Duality for finite p-algebras}
Priestley~\cite{Pri75} gave a topological duality for p-algebras.
As usual, for finite objects the topology is discrete, so for our purposes
the non-topological part will suffice. Objects dual to finite p-algebras are
finite posets, and morphisms are order-preserving maps 
$f\colon X\to Y$ satisfying the condition $f(\max\up{x}) = \max\up{f(x)}$
for every $x\in X$. We call them \emph{pp-morphisms}, by analogy
to, and to avoid confusion with, p-morphisms familiar from intuitionistic
and modal logics. 

For a p-algebra $\alg{A}$, we let $\mathcal{F}_p(\alg{A})$ stand for the
poset of prime filters of $\alg{A}$ ordered by inclusion. If $\alg{A}$ is
finite, $\mathcal{F}_p(\alg{A})$ is isomorphic to the poset of join-irreducible
elements of $A$ ordered by the converse of the order inherited from $\alg{A}$.
We write $\delta(\alg{A})$ for that poset, and for any homomorphism
$h\colon\alg{A}\to\alg{B}$ we define 
$\delta(h)\colon \delta(\alg{B})\to \delta(\alg{A})$ by $\delta(h)(F) =
h^{-1}(F)$ for any $F\in\mathcal{F}_p(\alg{B})$.
So defined $\delta(h)$ is a pp-morphism.

Conversely, for a finite poset $X$, we let $\mathrm{Up}(X)$
stand for the set of upsets of $X$. The algebra 
$\varepsilon(X) = (\mathrm{Up}(X); \cup,\cap,{}^*, \emptyset,X)$,
with $U^* \deq X\setminus \dw{U}$ is a p-algebra.
For a pp-morphism $f\colon X\to Y$, we define
$\varepsilon(f)\colon \varepsilon(Y)\to \varepsilon(X)$ by
$\varepsilon(f)(U) \deq \up{f^{-1}(U)}$ for any $U\in\mathrm{Up}(\mathbb{Y})$.
So defined $\varepsilon(f)$ is a homomorphism of p-algebras.

As expected, for a finite p-algebra $\alg{A}$ we have $\alg{A} \cong
\varepsilon(\delta(\alg{A}))$, and for a finite poset $X$ we have
$X \cong \delta(\varepsilon(X))$. The usual dual correspondences between
surjective and injective maps hold, as well as the correspondence between
finite direct products and finite disjoint unions. 

By Lemma~\ref{lem:fin-algs-in-qvars} and duality, we immediately obtain
the next result, which will be used later on without further notice.
The reader can also find it in~\cite{GLQ80}, as Lemma~9. 

\begin{lemma}\label{lem:fin-posets}
Let $\mathcal{P}$ be a family of finite posets, and let $\mathbb{X}$ be a finite
poset. Then $\varepsilon(\mathbb{X})\in
Q\left(\{\varepsilon(\mathbb{P}): \mathbb{P}\in \mathcal{P}\}\right)$ if and only if
for some $\mathbb{P}_1,\dots, \mathbb{P}_n\in\mathcal{P}$ there is a surjective
pp-morphism $f\colon \mathbb{P}_1\uplus\dots\uplus\mathbb{P}_n\to\mathbb{X}$.
\end{lemma}  

\section{Subquasivariety lattice}
In general, for a quasivariety $\mathcal{Q}$, we write $L^q(\mathcal{Q})$ for
the lattice of subquasivarieties of $\mathcal{Q}$, but our focus in this section
will be on $L^q(\var{Pa})$ and $L^q(\var{Pa}_n)$. As we already mentioned,
in~\cite{Dzi85} and~\cite{AD94} it is shown that $L^q(\var{Pa})$ it is
Q-universal. In fact this result is stronger, since
the proof of the fact that the conditions for Q-universality are satisfied
by $\var{Pa}$ applies without any changes to $\var{Pa}_n$ for any $n\geq 3$.
To see it, note that the proof from~\cite{Dzi85} uses a construction
whose building blocks are Steinerian quasigroups and the corresponding posets
(see the next subsection). Inspection quickly shows that these constructions 
are carried out within $\var{Pa}_3$.  

\begin{prop}[Adams, Dziobiak]
For $n\geq 3$ the variety $\var{Pa}_n$ is Q-universal.
\end{prop}  

It follows in particular that $L^q(\var{Pa}_n)$ is uncountable for any
$n\geq 3$. We will improve this result slightly, showing that for every $n\geq 3$ 
the interval $[\var{Pa}_{n-1}, \var{Pa}_n)\subsetneq L^q(\var{Pa}_n)$ is uncountable
and has the largest member. Our techniques are, again, a modification
of the ``Steinerian'' constructions of~\cite{GLQ80}, but
before we apply them, we say a few more things about $L^q(\var{Pa})$.

\begin{lemma}\label{lem:gen-as-quasiv}
For each $m \in\mathbb{N}$ the variety $\var{Pa}_m$ is generated
by $\Bnalg{m}$ as a quasivariety. Therefore, for any 
quasivariety $\mathcal{Q}\subseteq \var{Pa}$ we have
$$
\var{Pa}_m\subseteq\mathcal{Q} \iff \Bnalg{m}\in\mathcal{Q}.
$$
\end{lemma}

\begin{proof}
The subdirectly irreducible algebras in $\var{Pa}_m$ are
precisely the algebras $\Bnalg{k}$ for $k\leq m$. Hence, by subdirect
representation, every algebra $\alg{A}\in\var{Pa}_m$ belongs
to $SP\{\Bnalg{k}: k\leq m\}$; since $\Bnalg{k} \leq \Bnalg{m}$ for $k\leq m$,
we have $\alg{A}\in SPS(\Bnalg{m}) = SP(\Bnalg{m})$.
\end{proof}

Next, we will see that the property $\Bnalg{m}\notin\mathcal{Q}$ can be
characterised by a quasiequation, and therefore there exists the largest
quasivariety not containing $\mathcal{Q}$. This property of $L^q(\var{Pa})$
can be described in purely lattice-theoretic terms as a \emph{splitting}.
We say that a pair of elements $(a,b)\in L^2$ splits a lattice $L$
if (i) $a\not\leq b$, and (ii) $a\leq c \text{ or } c\leq b$ holds for any $c\in
L$. Splittings in lattices of varieties and quasivarieties are an important
divide-and-conquer method of analysing these lattices (see McKenzie~\cite{McK72} 
for a classic in the theory). From general results of Wroński~\cite{Wro81}
it follows that each $\var{Pa}_m$ \emph{splits} the lattice of
subquasivarieties of $\var{Pa}$, or in other words that $\Bnalg{m}$ is a
\emph{splitting algebra} in $\var{Pa}$. Indeed, it suffices to exhibit
a quasiequation axiomatising the largest quasivariety that does not contain
$\var{Pa}_m$. One way of obtaining such a quasi-identity is to take
$$
D^+(\Bnalg{m})\implies (x_e = x_1),
$$
where $D^+(\Bnalg{m})$ is the positive unnested diagram of
$\Bnalg{m}$. Wroński used essentially this technique to obtain his results
announced in~\cite{Wro81}. We will however take a more
scenic, roundabout route. 

First we exhibit another set of quasi-identities, also axiomatising
the splitting companions of $\var{Pa}_m$ in $L^q(\var{Pa})$. They are shorter, and
easier to work with.  
\begin{equation*}
\bigFOand_{1\leq i\leq n} (x^*_i = \bigvee_{j\neq i} x_j)
\tag{$\mathbf{qb}_n$}\ \implies \bigvee_{1\leq i\leq n} x_i = 1.
\end{equation*}

\begin{lemma}\label{lem:charact-Bn}
For all $n\geq 1$, a p-algebra $\alg{A}$ satisfies \textup($\mathbf{qb}_n$\textup)
if and only if\/ $\Bnalg{n}\not\leq\alg{A}$.
\end{lemma}

\begin{proof}
Note that $\Bnalg{n}\not\models \mathbf{qb}_n$. To see it, evaluate
$x_1,\dots,x_n$ to the $n$ atoms of $B_n$. Then the antecedent is clearly
satisfied, but the consequent evaluates to $e$. Hence, every $\alg{A}$ with
$\Bnalg{n}\leq\alg{A}$ has $\alg{A}\not\models
\mathbf{qb}_n$. This proves the forward direction.
For the converse, let $\alg{A}\not\models\mathbf{qb}_n$ and
consider a falsifying valuation with $v(x_i) = a_i$.

\noindent
\emph{Claim 0.\/} For all $i\neq j$ we have 
$a_i\wedge a_j = 0$. To see it, note that
$a^*_i = \bigvee_{j\neq i} a_j$ implies
$$
0 = a_i\wedge a^*_i = a_i\wedge \bigvee_{j\neq i} a_j 
= \bigvee_{j\neq i} (a_i\wedge a_j)
$$
and hence $a_i\wedge a_j = 0$ as claimed.

\noindent
\emph{Claim 1.\/} The set $\{a_1,\dots,a_n\}$ is an antichain. 
For suppose
$a_{i_0}\leq a_{j_0}$. As $a_{i_0} \wedge a_{j_0} = 0$ we have $a_{i_0} = 0$, so
$a^*_{i_0} = 1$ and so by the antecedent
$\bigvee_{i\neq i_0} a_i = 1$. Hence, $\bigvee_{1\leq i\leq n} a_i = 1$,
so $\mathbf{qb}_n$ holds, contrary to the assumption. This proves the claim.

\noindent
\emph{Claim 2.\/} For any $I,J\subseteq\{1,\dots,n\}$, if $I\neq J$ then
$\bigvee_{i\in I} a_i \neq \bigvee_{j\in J} a_j$. Suppose otherwise and let
$i_0\in I\setminus J$. Then
$$
a_{i_0} = a_{i_0}\wedge \bigvee_{i\in I} a_i = 
a_{i_0}\wedge \bigvee_{j\in J} a_j = \bigvee_{j\in J}(a_{i_0}\wedge a_j) = 0
$$
which contradicts Claim~1.

\noindent
\emph{Claim 3.\/} For every $i$ we have $a^{**}_i = a_i$.
By the antecedent of ($\mathbf{qb}_n$), using distributivity of $*$ over join, we get
$$
a^{**}_i = \bigl(\bigvee_{j\neq i} a_j\bigr)^* =
\bigwedge_{j\neq i} a^*_j = \bigwedge_{j\neq i}\bigvee_{k\neq j} a_k
$$
Further, by distributivity
$$
\bigwedge_{j\neq i}\bigvee_{k\neq j} a_k = 
\bigvee_f\bigwedge_{s\neq i} a_{f(s)}
$$
where $f$ ranges over choice functions of the family $\{N_r: 1\leq r\leq n\}$
with $N_r = \{1,\dots,n\}\setminus\{r\}$. The only non-zero meet
$\bigwedge_{s\neq i} a_{f(s)}$ comes from the constant function $f(s) = i$,
and so we get the desired
$$
\bigvee_f\bigwedge_{s\neq i} a_{f(s)} = a_i.
$$

Now, combining Claims~1, 2, 3 and distributivity of $*$ over join, it is
straightforward to show that
$(\bigvee_{i\in I} a_i)^* = \bigvee_{i\notin I}a_i$ for every non-empty
$I\subseteq \{1,\dots, n\}$; in particular $(\bigvee_{1\leq i\leq n} a_i)^* = 0$.   
Since $\alg{A}\not\models \mathbf{qb}_n$ we must have
$\bigvee_{1\leq i\leq n} a_i < 1$, and therefore the subalgebra
of $\alg{A}$ generated by $\{a_1,\dots,a_n\}$ is isomorphic to
$\Bnalg{n}$. 
\end{proof}

Note that $\mathbf{qb}_1$ excludes $\Bnalg{1}$, so p-algebras satisfying
it are precisely Boolean algebras, that is, $\mathrm{Mod}(\mathbf{qb}_1) = \var{Pa}_0$.
Analogously, $\mathbf{qb}_2$ excludes $\Bnalg{2}$, so p-algebras satisfying
it are precisely Stone algebras, that is, $\mathrm{Mod}(\mathbf{qb}_2) = \var{Pa}_1$.
For $n\geq 3$ this simple picture breaks down. We have work to do.

\begin{lemma}\label{lem:greatest-qvar}
The quasivariety $\mathrm{Mod}(\mathbf{qb}_n)$ is the splitting companion of\/
$\var{Pa}_n$ in $L^q(\var{Pa})$. Hence $\mathrm{Mod}(\mathbf{qb}_n,
\mathbf{ib}_n)$ is the largest proper subquasivariety of\/ $\var{Pa}_n$. 
\end{lemma}  

\begin{proof}
By Lemmas~\ref{lem:width}, \ref{lem:gen-as-quasiv} and~\ref{lem:charact-Bn}.
\end{proof}

We will write $\var{Pa}_n^-$ for $\mathrm{Mod}(\mathbf{qb}_n)\cap\var{Pa}_n$.
Analogously, we will write $\mathrm{Mod}(\mathbf{qb}_n)^+$ for 
$\mathrm{Mod}(\mathbf{qb}_n)\vee\var{Pa}_n$. Figure~\ref{fig:LQ-Pa} depicts
$L^q(\var{Pa})$. It turns out that the quasivarieties
generated by (sufficiently large) free p-algebras all belong to  
$\mathrm{Mod}(\mathbf{qb}_3)$. To prove it we will need a few facts
about the structure of free p-algebras. For the proof of the next
lemma we refer the reader to~\cite{KS24}.

\begin{lemma}\label{lem:under-each}
Let $\alg{F}_m(k)$ be the $k$-generated free algebra in $\var{Pa}_m$, and let
$\alg{F}(k)$ be the $k$-generated free algebra in $\var{Pa}$.
The following hold:
\begin{enumerate}
\item The set $\max(\delta(\alg{F}_m(k)))$ has $2^k$ elements.
\item For any $A\subseteq \max(\delta(\alg{F}_m(k)))$ with $0< |A|\leq m$ there
  is a $u\in \delta(\alg{F}_m(k))$ such that $\max(\up{u}) = A$.
\item If $m\geq 2^k$, then  $u\in \delta(\alg{F}_m(k))$ with
  $\max(\up{u}) = \max\bigl(\delta(\alg{F}_m(k))\bigr)$ 
is unique, and it is the bottom element of $\delta(\alg{F}_m(k))$.  
\item If $m\geq 2^k$, then $1^{\alg{F}_m(k))}$ is join-irreducible; hence
$\{1^{\alg{F}_m(k))}\}$ is a prime filter of $\alg{F}_m(k))$.  Moreover,
$\alg{F}_m(k)  = \alg{F}(k)$.
\end{enumerate}
\end{lemma}

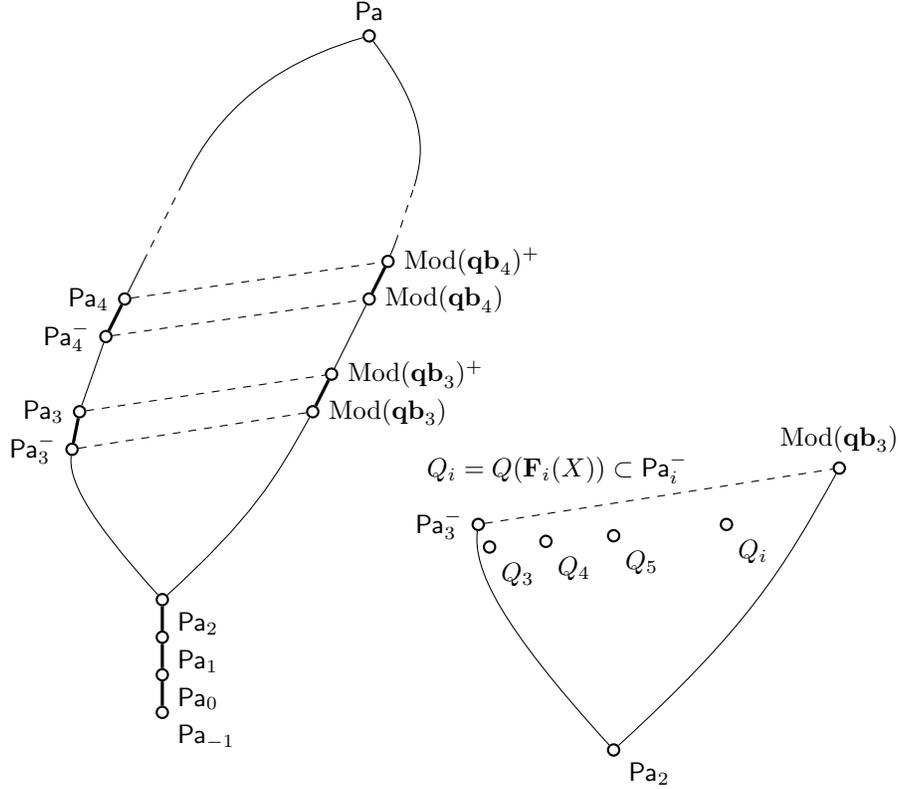
\begin{figure}[h]
\begin{tikzpicture}
\begin{scope}    
\node[dot, label=-30:$\var{Pa}_{-1}$] (v1) at (0,0.5) {};
\node[dot, label=-30:$\var{Pa}_{0}$] (v2) at (0,1) {};
\node[dot, label=-30:$\var{Pa}_{1}$] (v3) at (0,1.5) {};
\node[dot, label=-30:$\var{Pa}_{2}$] (v4) at (0,2) {};
\node[dot, label=left:$\var{Pa}_{3}^-$] (v5) at (-1.2,4) {};
\node[dot, label=left:$\var{Pa}_{3}$] (v6) at (-1.1,4.5) {};
\node[dot, label=left:$\var{Pa}_{4}^-$] (v7) at (-0.75,5.5) {};
\node[dot, label=left:$\var{Pa}_{4}$] (v8) at (-0.5,6) {};
\node[coordinate] (v13) at (-0.25,6.5) {};
\node[coordinate] (v14) at (0.25,7.5) {};
\node[dot, label=right:$\mathrm{Mod}(\mathbf{qb}_3)$] (v9) at (2,4.5) {};
\node[dot, label=right:$\mathrm{Mod}(\mathbf{qb}_3)^+$] (v10) at (2.25,5) {};
\node[dot, label=right:$\mathrm{Mod}(\mathbf{qb}_4)$] (v11) at (2.75,6) {};
\node[dot, label=right:$\mathrm{Mod}(\mathbf{qb}_4)^+$] (v12) at (3,6.5) {};
\node[coordinate] (v15) at (3.1,6.75) {};
\node[coordinate] (v16) at (3.35,7.5) {};
\node[dot, label=above:$\var{Pa}$] (v17) at (2.75,9.5) {};

\path[draw,very thick] (v1)--(v2)--(v3)--(v4);
\path[draw] (v5)--(v6)--(v7)--(v8)--(v13);
\draw[dashed] (v13)--(v14) (v15)--(v16);
\path[draw] (v9)--(v10)--(v11)--(v12)--(v15);
\draw[very thick] (v5)--(v6) (v7)--(v8) (v9)--(v10) (v11)--(v12);
\draw[dashed] (v5)--(v9) (v6)--(v10) (v7)--(v11) (v8)--(v12);

\draw (v4).. controls (-1,3.1) and (-1.25,3.6) .. (v5);
\draw (v4).. controls (1.2,3.1) and (1.5,3.6) .. (v9);

\draw (v14).. controls (0.5,8) and (1,9) .. (v17);
\draw (v16).. controls (3.5,8) and (3.5,8.5) .. (v17);
\end{scope}%
\begin{scope}[scale=1.5, xshift=4cm, yshift=-2cm]
\node[dot, label=-30:$\var{Pa}_{2}$] (v4) at (0,2) {};
\node[dot, label=left:$\var{Pa}_{3}^-$] (v5) at (-1.2,4) {};
\node[coordinate] (v13) at (-0.25,6.5) {};
\node[coordinate] (v14) at (0.25,7.5) {};
\node[dot, label=above:$\mathrm{Mod}(\mathbf{qb}_3)$] (v9) at (2,4.5) {};
\node[coordinate] (v15) at (3.1,6.75) {};
\node[coordinate] (v16) at (3.35,7.5) {};

\node[dot, label=-70:$Q_3$] (v1) at (-1.1,3.8) {};
\node[dot, label=-70:$Q_4$] (v2) at (-0.6,3.85) {};
\node[dot, label=-70:$Q_5$] (v3) at (0,3.9) {};
\node[dot, label=-70:$Q_i$] (v7) at (1,4) {};

\draw[dashed] (v5)--(v9);

\draw (v4).. controls (-1,3.1) and (-1.25,3.6) .. (v5);
\draw (v4).. controls (1.2,3.1) and (1.5,3.6) .. (v9);
\node (v6) at (-0.5,4.5) {$Q_i = Q(\alg{F}_i(X))\subset \var{Pa}_i^-$};
\end{scope}
\end{tikzpicture}\label{fig:LQ-Pa} 
\caption{Left: the lattice of subquasivarieties of $\var{Pa}$; thick lines indicate
  coverings. Right: the interval $[\var{Pa_2},\mathrm{Mod}(\mathbf{qb}_3)]$
  containing all quasivarieties generated by free algebras.} 
\end{figure}

\begin{lemma}\label{lem:free-qb3}
For any $m\geq 3$ and any $k\geq 2$,
the algebra $\mathbf{F}_m(k)$ satisfies $(\mathbf{qb}_3)$.
\end{lemma}

\begin{proof}
By Lemma~\ref{lem:charact-Bn} and duality, it suffices to show that
there is no surjective pp-morphism 
$f\colon \delta(\mathbf{F}_m(k))\to \delta(\Bnalg{3})$. Assume
\emph{a contrario} that one exists. Let $\max(\delta(\Bnalg{3})) = \{p,q,r\}$.
By pp-morphism conditions, $f\bigl(\max(\delta(\mathbf{F}_m(k)))\bigr) = \{p,q,r\}$. 
Put $A = \max(f^{-1}(p))$, $B = \max(f^{-1}(q))$, $C = \max(f^{-1}(r))$, so that
$A\uplus B\uplus C = \max(\delta(\mathbf{F}_m(k)))$. 
Then by Lemma~\ref{lem:under-each}(1) we have 
$|A| + |B| + |C| = 2^k > 3$. Without loss of generality assume $|A|\geq 2$, and
pick elements $a_1,a_2\in A$ and $b\in B$. By Lemma~\ref{lem:under-each}(2),
there is a $u\in \delta(\alg{F}_m(k))$ with $\max(\up{u}) = \{a_1,a_2,b\}$. 

Now consider $f(u)$. Clearly $f(u)\notin \{p,q,r\}$ so $f(u)$ is the bottom
element of $\delta(\Bnalg{3})$, so $\max(\up{f(u)}) = \{p,q,r\}$.
On the other hand, $\max(\up{f(u)}) = f(\max(\up{u})) = \{f(a_1),f(a_2),f(b)\} =
\{p,q\}$, which is a contradiction. 
\end{proof}

This immediately implies the next
result, which we found somewhat curious.

\begin{theorem}\label{thm:free-algs}
For any $m\geq 3$ and any set of generators $X$ with $|X|\geq 2$, the algebra
$\mathbf{F}_m(X)$, free in $\var{Pa}_m$, belongs to 
$\mathrm{Mod}(\mathbf{qb}_3)$. The same holds for the algebra
$\mathbf{F}(X)$, free in $\var{Pa}$. 
\end{theorem}

Graetzer \emph{et al.}~\cite{GLQ80} prove 
that the interval $(\var{Pa}_3]$ is uncountable. By
Lemma~\ref{lem:greatest-qvar} and the fact $|L^q(\var{Pa}_2)| = 4$ we
have that $[\var{Pa}_2, \var{Pa}_3^-]$ is uncountable.
We will prove the same for any interval $[\var{Pa}_{m-1}, \var{Pa}_{m}^-]$
with $m\geq 3$.
 
\subsection{Steiner quasigroups and quasivarieties}
Following~\cite{GLQ80}, we say that a \emph{Steiner quasigroup} is an
algebra $\alg{S} = (S,\cdot)$ satisfying the identities
\begin{enumerate}
\item $x^2 = x$,
\item $xy = yx$,
\item $x(xy) = y$.
\end{enumerate}  
For any Steiner quasigroup $\alg{S}$ we define
$S_3 = \{\{x,y,xy\}: x,y\in S, x\neq y\}$. So defined $S_3$ consists of all
3-element subalgebras of $\alg{S}$, and $(S,S_3)$ is a Steiner triple system.
Next, we let $P(\alg{S}) \deq (S\cup S_3; \leq)$ with $\leq$ defined by
putting $x\leq y$ if $x=y$ or $x\in S_3, y\in S, y\in x$. Clearly,
$P(\alg{S})$ is a poset, $\min P(\alg{S}) = S_3$,
$\max P(\alg{S}) = S$, $P(\alg{S}) = S\uplus S_3$, and
$|\max[u)| = 3$ for any $u\in S_3$.  

A Steiner quasigroup $\alg{S}$ is called \emph{planar} if $|S|\geq 4$ and
any three distinct elements $a,b,c$, such that $\{a,b,c\}$ is not a subalgebra of
$\alg{S}$, generate $\alg{S}$. Doyen~\cite{Doy70} showed that
planar Steiner quasigroups exists for each $n\geq 7$ with
$n = 1\text{ or } 3 \pmod 6$. Quackenbush~\cite{Qua76} proved that
for $n\neq 9$ any planar Steiner quasigroup of cardinality $n$ is simple.
The next lemma (cf.~\cite{GLQ80} Lemma~10) is an immediate consequence. 

\begin{lemma}\label{lem:PSQ-simple}
If $\alg{S}$ and $\alg{R}$ are non-isomorphic finite planar Steiner
quasigroups, and $f:\alg{S}\to\alg{R}$ is a homomorphism, then
$|f(S)| = 1$.
\end{lemma}

Let $\alg{Q}$ be the 7-element planar Steiner quasigroup, and for any $m \geq 3$ let
$P_1^m$ be the poset dual to $\Bnalg{m}$, in which the maximal elements
are $a_1, a_2,\dots, a_m$. We define $W_m = P(\alg{Q})\mathbin{\diamondsuit}
P_1^m$ to be the horizontal pasting of $P(\alg{Q})$ and $P_1^m$ identifying two distinct
maximal elements of $P(\alg{Q})$ with $a_1$ and $a_2$, so that the universe
of $W_m$ is $P(\alg{Q})\cup P_1^m$ with $P(\alg{Q})\cap P_1^m = \{a_1,a_2\}$.
Then, the ordering on $W_m$ is $\leq^{W_m}\ =\ \leq^{P(\alg{Q})}\cup
\leq^{P_1^m}$. The construction of $W_m$ is illustrated in Figure~\ref{fig:constr}. 

\begin{figure}
\begin{center}
\begin{tikzpicture}
\node[dot, label=above:$a_2$] (v0) at (6,1.5) {};
\node[dot, label=above:$a_1$] (v1) at (5,1.5) {};
\node[dot, label=above:$c_1$] (v2) at (4,1.5) {};
\node[dot, label=above:$c_2$] (v3) at (3,1.5) {};
\node[dot, label=above:$c_3$] (v4) at (2,1.5) {};
\node[dot, label=above:$c_4$] (v5) at (1,1.5) {};
\node[dot, label=above:$c_5$] (v6) at (0,1.5) {};

\node[dot] (v7) at (6,0) {};
\node[dot] (v8) at (3,0) {};
\node[dot] (v9) at (5,0) {};
\node[dot] (v10) at (4,0) {};
\node[dot] (v11) at (2,0) {};
\node[dot] (v12) at (1,0) {};
\node[dot] (v13) at (0,0) {};

\path[draw,thick] (v0)--(v7) (v1)--(v7) (v2)--(v7); 
\path[draw,thick] (v0)--(v8) (v3)--(v8) (v4)--(v8); 
\path[draw,thick] (v1)--(v9) (v3)--(v9) (v5)--(v9); 
\path[draw,thick] (v2)--(v10) (v3)--(v10) (v6)--(v10); 
\path[draw,thick] (v1)--(v11) (v4)--(v11) (v6)--(v11); 
\path[draw,thick] (v2)--(v12) (v4)--(v12) (v5)--(v12); 
\path[draw,thick] (v0)--(v13) (v5)--(v13) (v6)--(v13); 

\node[] (v14) at (3,-1) {$P(\alg{Q})$};
\end{tikzpicture}  
\hspace{1cm}
\begin{tikzpicture}
\node[dot, label=below:$\bot$] (v0) at (0,0) {};
\node[dot, label=above:${a_1}$] (v1) at (-2,1.5) {};
\node[dot, label=above:${a_2}$] (v2) at (-1,1.5) {};
\node[dot, label=above:${a_3}$] (v3) at (0,1.5) {};
\node[dot, label=above:${a_m}$] (v4) at (2,1.5) {};
\node[] (v5) at (1,1.5) {$\dots$};
\path[draw,thick] (v0)--(v1) (v0)--(v2) (v0)--(v3) (v0)--(v4);
\node[] (v6) at (0,-1) {$P^m_1$};
\end{tikzpicture}

\vspace{1cm}
\begin{tikzpicture}
\node[dot, label=above:$a_2$] (v0) at (6,1.5) {};
\node[dot, label=above:$a_1$] (v1) at (5,1.5) {};
\node[dot, label=above:$c_1$] (v2) at (4,1.5) {};
\node[dot, label=above:$c_2$] (v3) at (3,1.5) {};
\node[dot, label=above:$c_3$] (v4) at (2,1.5) {};
\node[dot, label=above:$c_4$] (v5) at (1,1.5) {};
\node[dot, label=above:$c_5$] (v6) at (0,1.5) {};

\node[dot] (v7) at (6,0) {};
\node[dot] (v8) at (3,0) {};
\node[dot] (v9) at (5,0) {};
\node[dot] (v10) at (4,0) {};
\node[dot] (v11) at (2,0) {};
\node[dot] (v12) at (1,0) {};
\node[dot] (v13) at (0,0) {};

\path[draw,thick] (v0)--(v7) (v1)--(v7) (v2)--(v7); 
\path[draw,thick] (v0)--(v8) (v3)--(v8) (v4)--(v8); 
\path[draw,thick] (v1)--(v9) (v3)--(v9) (v5)--(v9); 
\path[draw,thick] (v2)--(v10) (v3)--(v10) (v6)--(v10); 
\path[draw,thick] (v1)--(v11) (v4)--(v11) (v6)--(v11); 
\path[draw,thick] (v2)--(v12) (v4)--(v12) (v5)--(v12); 
\path[draw,thick] (v0)--(v13) (v5)--(v13) (v6)--(v13);

\node[dot, label=below:$\bot$] (w0) at (7,0) {};
\node[dot] (w1) at (5,1.5) {};
\node[dot] (w2) at (6,1.5) {};
\node[dot, label=above:${a_3}$] (w3) at (7,1.5) {};
\node[dot, label=above:${a_m}$] (w4) at (9,1.5) {};
\node[] (w5) at (8,1.5) {$\dots$};
\path[draw,thick] (w0)--(w1) (w0)--(w2) (w0)--(w3) (w0)--(w4);

\node[] (v14) at (4,-1) {$W_m = P(\alg{Q})\mathbin{\diamondsuit} P^m_1$};
\end{tikzpicture}
\end{center}
\caption{Construction of $P(\alg{Q})\mathbin{\diamondsuit} P^m_1$}
\label{fig:constr}
\end{figure}

Choose a sequence $\{\alg{S}_k\}_{k\in \mathbb{N}}$ of planar Steiner quasigroups so
that $k\geq 2$ and $|S_k| = 6k+1$.
For any $K\subseteq \mathbb{N}\setminus\{0,1\}$ let
$\var{Q}_K^m$ be the quasivariety generated by 
the set $\{\varepsilon(P(\alg{S}_k))\}_{k\in K}\cup \{\varepsilon(W_m)\}$.
We have the following result.

\begin{lemma}\label{lem:techne}
  Let $\{\alg{S}_k\}_{k\in \mathbb{N}}$ be as above, and let $m>3$. Then
\begin{enumerate}  
\item For any $K\subseteq \mathbb{N}\setminus\{0,1\}$ we have
$\var{Q}_K^m \in [\var{Pa}_{m-1},\var{Pa}_{m}^-]$.
\item For any $K_1,K_2\subseteq \mathbb{N}\setminus\{0,1\}$ we have
$K_1\neq K_2\implies \var{Q}_{K_1}^m \neq \var{Q}_{K_2}^m$.  
\end{enumerate}  
\end{lemma}

\begin{proof}
For (1) we first show that $\Bnalg{m-1}\in S(\varepsilon(W_m))$. Let the maximal
elements of $P_1^{m-1}$ be $a_1',\dots, a_{m-1}'$ and let the smallest element
be $\bot'$. The map $h: W_m\to P_1^{m-1}$ given by
$$
h(u) = \begin{cases}
a'_1 & \text{ if } u\in P(\alg{Q})\\
a'_{i-1} & \text{ if } u = a_i \text{ and } 3\leq i\leq m\\
\bot'   &  \text{ if } u = \bot 
\end{cases}
$$
is a surjective pp-morphism. Hence $\Bnalg{m-1}\in S(\varepsilon(W_m))$.

Next we show that $\var{Q}_K^m\models\mathbf{qb}_m$. By Lemma~\ref{lem:charact-Bn} 
it suffices to show that  
$\Bnalg{m}\notin \var{Q}_K^m$. Suppose the contrary.
Then, by Lemma~\ref{lem:fin-posets} and duality there exists a finite
$K_0\subseteq K$ and a surjective pp-morphism
$$
h\colon \biguplus_{k\in K_0} P(\alg{S}_k)\uplus W_m \to P_1^m.
$$
Then $h(\max(\biguplus_{k\in K_0} P(\alg{S}_k)\uplus W_m)) = \max(P_1^m) =
\{a_1,\dots, a_m\}$, and moreover for some
$x\in \biguplus_{k\in K_0} P(\alg{S}_k)\uplus W_m$ we have
$h(x) = \bot$. If $x\in \biguplus_{k\in K_0} P(\alg{S}_k)\uplus P(\alg{Q})$, then
$|\max\up{x}| \leq 3$, but $|\max\up{h(x)}| > 3$ giving a contradiction.
Therefore $x = \min(P^m_1)$. It follows that
$h(\max\up{x}) = h(\{a_1,\dots, a_m\}) = \max(P_1^m)$, and therefore
$h(a_1)\neq h(a_2)$. Now, by construction of $W_m$ we have  that
$a_1,a_2\in\max(P(\alg{Q}))$, so using the fact that $\max(P(\alg{Q}))$ 
is $\alg{Q}$, we take $t\deq \{a_1,a_2, a_1a_2\}$.  Then
$t\in \min(P(\alg{Q}))$ and $\max\up{t} = \{a_1,a_2, a_1a_2\}$, so it follows that
$|h(\max\up{t})|\in\{2,3\}$.  But $h(t)$ is either a maximal element or
$h(t) = \min(P^m_1)$; in the former case $|\up{h(\max(t))}| = 1$, and in the
latter $|\up{h(\max(t))}| = m > 3$, yielding a contradiction in either case.
This proves (1).

For (2) let $K_1, K_2\subseteq \mathbb{N}\setminus\{0,1\}$, with $K_1\neq
K_2$. Without loss of generality, take $r\in K_1\setminus K_2$. We will show
that $\varepsilon(P(\alg{S}_r)\notin \var{Q}_{K_2}^m$. Suppose the contrary;
then as in the proof of (1) there exists
$K_0\subseteq K_2$ and a surjective pp-morphism
$$
h\colon \biguplus_{k\in K_0} P(\alg{S}_k)\uplus W_m \to P(\alg{S}_r).
$$
It is immediate that if $f\colon P_1\uplus P_2 \to G$ is a pp-morphism, then
$f|_{P_i}\colon P_i\to G$ is also a pp-morphism, for $i=1,2$. Therefore,
$h|_{P(\alg{S}_k)}\colon P(\alg{S}_k) \to P(\alg{S}_r)$ is a pp-morphism for every
$k\in K_0$, and $h|_{W_m}\colon W_m \to P(\alg{S}_r)$ is a pp-morphism as well. 
Moreover, it is easy to check that $h|_{P(\alg{Q})}: P(\alg{Q}) \to P(\alg{S}_r)$
is also a pp-morphism. In particular,
$h(\alg{S}_k) = h(\max(P(\alg{S}_k))) \subseteq \max(P(\alg{S}_r)) = \alg{S}_r$
and $h(\alg{Q})\subseteq \alg{S}_r$. 
By construction of $P(\alg{S}_k)$ and $P(\alg{Q})$ it follows that the maps 
$h|_{\alg{S}_k}\colon \alg{S}_k\to \alg{S}_r$ and
$h|_{\alg{Q}}\colon \alg{Q}\to \alg{S}_r$ are quasigroup homomorphisms.
Hence, by Lemma~\ref{lem:PSQ-simple} we get  $|h(\alg{S}_k)| =
1 = |h(\alg{Q})|$. 

Next, take $a,b,\in \alg{S}_r$ with $a\neq b$. Then
$t = \{a,b, ab\} \in \min(P(\alg{S}_r))$, and since $h$ is a surjective
pp-morphism there exists some $x\in \biguplus_{k\in K_0} P(\alg{S}_k)\uplus W_m$
such that $h(x) = t$. Moreover, by pp-morphism conditions
$\max(\up{t}) = h(\max(\up{x}))$, and
$|\max(\up{t})| = 3$. Now, if $x\in \min(\biguplus_{k\in K_0} P(\alg{S}_k))
\uplus \min(P(\alg{Q}))$, then we would have
$|h(\max(\up{x}))| = 1$ contradicting $|\max(\up{t})| = 3$. Hence,
$x$ is the smallest element of $P^m_1$, that is, $x = \bot$. Pick any element
$t'\in\min P(\alg{S}_r)$ with $t'\neq t$ (there are many such elements: it
suffices to take $c\notin \{a,b,ab\}$ and let $t' = \{a,c,ac\}$). 
Repeating the reasoning above, we conclude that $h(\bot) = t'$,
which yields a
contradiction and ends the proof. 
\end{proof}  

\begin{theorem}\label{thm:cont-everywhere}
For any $m\geq 3$, the interval\/ $[\var{Pa}_{m-1},\var{Pa}^-_{m}]$ in the lattice of
subquasivarieties of\/ $\var{Pa}$  is of cardinality continuum.
\end{theorem}  

\begin{proof}
For $m = 3$, it was shown in~\cite{GLQ80}. For $m>3$, it is   
immediate by Lemma~\ref{lem:techne}.
\end{proof}   

\section{Structural (in)completeness}

A deductive system $\mathcal{S}$ is \emph{structurally complete}
if all its admissible rules are derivable. 
The notion of structural completeness was introduced by Pogorzelski
in~\cite{Pog71}, and investigated by many authors in various forms and
modifications. Pogorzelski and Wojtylak~\cite{PW08} is a monography of the subject.
Since deductive systems correspond, at least roughly, to quasivarieties, so an
analogous algebraic notion naturally arose, namely, a quasivariety $\mathcal{Q}$
is structurally complete if every quasiequation true in the countably generated
free algebra in $\mathcal{Q}$ remains true throughout $\mathcal{Q}$.
This was first explicitly formulated, as far the authors are aware, in
Bergman~\cite{Ber91}.

It is obvious that a  
quasivariety $\mathcal{Q}$ is structurally complete if and only if $\mathcal{Q}$
is generated by its free algebra on countably many generators, or equivalently,
by its finitely generated free algebras. Hence, if $\mathcal{V}$ is a variety, then
$\mathcal{V}$ is structurally complete if and only if $\mathcal{V}$ is generated by its
free algebras as a quasivariety.  

Algebraic investigations of structural completeness have recently (from the
authors' point of view) seen a revival, see for instance in~\cite{ORvA},
\cite{CM09}, \cite{RS16}, \cite{BM23}.
Certain weakenings of structural completeness were
investigated in~\cite{DS16}, \cite{Wro09}, and, in a different direction
in~\cite{Cit16}, \cite{Cit20}. For several recent results, and also excellent
surveys of algebraic methods in the area, we refer the reader to~\cite{AUg24}
and~\cite{AC24}.

For varieties of p-algebras, Theorem~\ref{thm:free-algs} has the
following immediate corollary, also stated in~\cite{KS24}.

\begin{cor}\label{cor:str-incompl}
The variety $\var{Pa}$ and varieties $\var{Pa}_n$, for every $n>2$, are not
structurally complete. 
\end{cor}

On the other hand, Mints~\cite{Min72} proved (the logical equivalent of) the following.

\begin{theorem}[Mints]\label{thm:mints}
Let $\varphi$ be a quasiequation of the form
$$
t_1=1\mathbin{\&} \dots \mathbin{\&} t_k=1 \implies s= 1,
$$
over $n$ variables, in the language of p-algebras, but considered as a
quasiequation over all Heyting algebras. Let $\alg{H}(n)$ be the $n$-generated
free Heyting algebra, and let $\var{HA}$ be the variety of Heyting algebras.
Then the following holds:
$$
\alg{H}(n)\models\varphi\quad \text{implies}\quad
\var{HA}\models\varphi.
$$
\end{theorem}

Corollary~\ref{cor:str-incompl} and Theorem~\ref{thm:mints} at first sight 
may seem to contradict each other. It is not the case, since
over p-algebras, unlike in Heyting algebras, an identity $t = s$ is in general
\emph{not} equivalent to any identity of the form $r = 1$. We will now show that
Theorem~\ref{thm:mints} can be derived from our description of free p-algebras.

\begin{theorem}\label{thm:special-str-compl}
Let $\varphi$ be a quasiequation of the form
$$
t_1=1\mathbin{\&} \dots \mathbin{\&} t_k=1 \implies s= 1,
$$
over $n$ variables, in the language of p-algebras. The following holds:
$$
\alg{F}(n)\models\varphi\quad \text{implies}\quad
\var{Pa}\models\varphi.
$$
\end{theorem}  

\begin{proof}
Suppose $\var{Pa}\not\models\varphi$. Then, some $n$-generated p-algebra
falsifies $\varphi$ and therefore $\Bnalg{2^n}\not\models\varphi$, since
$\Bnalg{2^n}$ is the largest $n$-generated subdirectly irreducible p-algebra. 
Let $\overline{a}$ be an $n$-tuple from $\Bnalg{2^n}$ such that
$t_i(\overline{a}) = 1$ for each $i\in\{1,\dots,k\}$ and
$s(\overline{a}) \neq 1$. Let $h\colon\alg{F}(n)\to \Bnalg{2^n}$ be the natural
homomorphism mapping generators to generators. The dual map
$\delta(h)\colon \delta(\Bnalg{2^n}) \to \delta(\alg{F}(n))$
is injective and maps maximal elements to maximal elements.
By Lemma~\ref{lem:under-each}(1), we have
$|\max(\delta(\alg{F}(n)))| = 2^n = |\max(\delta(\Bnalg{2^n}))|$, so 
$\delta(h)|_{\max(\delta(\Bnalg{2^n}))}$ is a bijection between
$\max(\delta(\Bnalg{2^n}))$ and $\max(\delta(\alg{F}(n)))$.
The bottom element of $\delta(\Bnalg{2^n})$ is $\{1^{\Bnalg{2n}}\}$, 
and by Lemma~\ref{lem:under-each}(3) we get that
$\delta(h)(\{1^{\Bnalg{2n}}\})$ is the bottom element of $\delta(\alg{F}(n))$.
By duality, it means that $h^{-1}(\{1^{\Bnalg{2n}}\}) = \{1^{\alg{F}(n)}\}$.
Since $h$ is surjective, for each $a_j\in\Bnalg{2^n}$ there is
a $p_j\in\alg{F}(n)$ such that $h(p_j) = a_j$. Therefore,
writing $\overline{p}$ for $(p_1,\dots,p_n)$, we obtain 
$t_i(\overline{p})\in h^{-1}(\{t_i(\overline{a})\})$.
But $h^{-1}(\{t_i(\overline{a})\}) = \left\{1^{\alg{F}(n)}\right\}$,
thus $t_i(\overline{p}) = 1^{\alg{F}(n)}$ for each $i$; yet
$s(\overline{p}) \neq 1^{\alg{F}(n)}$.
This however contradicts the
assumption $\alg{F}(n)\models\varphi$.
\end{proof}  

Note that Theorem~\ref{thm:special-str-compl} follows easily from
Theorem~\ref{thm:mints}, since $\alg{F}(k)$ is a subreduct of
$\alg{H}(k)$ for any $k$. We will show that the converse also holds.

\begin{lemma}\label{lem:our-implies-mints}
Let $\varphi$ be a quasiequation of Theorems~\ref{thm:mints}
and~\ref{thm:special-str-compl}. If the property
\begin{equation}\label{eq:we}
\alg{F}(n)\models\varphi\quad \text{implies}\quad
\var{Pa}\models\varphi
\end{equation}
holds, then the property 
\begin{equation}\label{eq:mints}
\alg{H}(n)\models\varphi\quad \text{implies}\quad
\var{HA}\models\varphi
\end{equation}
holds as well.
\end{lemma}  

\begin{proof}
Assume $\alg{H}(n)\models\varphi$. Then $\alg{H}(n)^r\models\varphi$, where
$\alg{H}(n)^r$ is the p-algebra reduct of $\alg{H}(n)$. Let
$f\colon \alg{F}(n)\to \alg{H}(n)^r$ be the natural p-algebra homomorphism
mapping generators to generators. Since $f$ is an injection, 
$\alg{F}(n)\models \varphi$ and so by~\eqref{eq:we} we have
$\var{Pa}\models\varphi$. Suppose for some $\alg{A}\in\var{HA}$ falsifies
$\varphi$, and let $\overline{a} = (a_1,\dots,a_n)$ be an $n$-tuple from $A$
witnessing it. Then, we have $t_i(\overline{a}) = 1$ for all $i\in\{1,\dots,k\}$ but
$s(\overline{a})\neq 1$. The same $n$-tuple witnesses
the failure of $\varphi$ on the p-algebra reduct $\alg{A}^r$ of
$\alg{A}$. But $\alg{A}^r\in\var{Pa}$ so $\var{Pa}\not\models\varphi$, yielding
a contradiction.
\end{proof}  

We finish off with some questions that seem interesting, at least to us. 
By Theorem~\ref{thm:free-algs} we know that all free p-algebras belong
to $\mathrm{Mod}(\mathbf{qb}_3$), but we know very little about
quasivarieties $Q(\alg{F}_m(X))$. Of course we have
$Q(\alg{F}_m(X))\subseteq \var{Pa}_m^-$, and for $m > 3$ the inclusion is
proper since $\Bnalg{3}\in\var{Pa}_m^-$ but
$\Bnalg{3}\notin Q(\alg{F}_m(X))$. 
We are nearly sure that it is also proper for $m = 3$.

\begin{question}  
Describe $Q(\alg{F}_m(X))$ for $m\geq 3$. Is 
$Q(\alg{F}_3(X))$ properly contained in $\var{Pa}_3^-$?
\end{question}

The quasivarieties $\mathrm{Mod}(\mathbf{qb}_m)$ also seem of interest. In particular
$\mathrm{Mod}(\mathbf{qb}_3)$, to which all free p-algebras belong. Clearly, we have
$H(\mathrm{Mod}(\mathbf{qb}_3)) = \var{Pa} = H(\{\alg{F}_m(X): m\in\mathbb{N}\})$,
so $V(\mathrm{Mod}(\mathbf{qb}_3)) = V(\{\alg{F}_m(X): m\in\mathbb{N}\})$.

\begin{question}
Does equality hold for quasivariety generation as well? In other words,
is\/ $\mathrm{Mod}(\mathbf{qb}_3)$ generated by free p-algebras?
\end{question}  

Another interesting question concerns the lower part of
$\mathrm{Mod}(\mathbf{qb}_3)$. We know that the interval
$[\var{Pa}_{-1},\var{Pa}_2]$ is a 4-element chain. Let $\mathbb{P}$ be any
of the posets below.
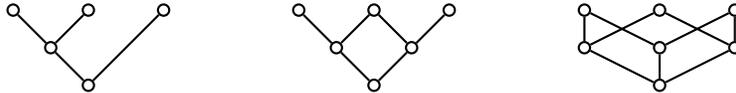
\begin{center}
\begin{tikzpicture}
\node[dot] (v0) at (0,0) {};
\node[dot] (v1) at (-0.5,0.5) {};
\node[dot] (v2) at (-1,1) {};
\node[dot] (v3) at (0,1) {};
\node[dot] (v4) at (1,1) {};
\path[draw,thick] (v0)--(v1) (v0)--(v4) (v1)--(v2) (v1)--(v3);
\end{tikzpicture}\hspace{1.5cm}
\begin{tikzpicture}
\node[dot] (v0) at (0,0) {};
\node[dot] (v1) at (-0.5,0.5) {};
\node[dot] (v2) at (-1,1) {};
\node[dot] (v3) at (0,1) {};
\node[dot] (v4) at (1,1) {};
\node[dot] (v5) at (0.5,0.5) {};
\path[draw,thick] (v0)--(v1) (v0)--(v5) (v1)--(v2) (v5)--(v3) (v1)--(v3) (v5)--(v4);
\end{tikzpicture}\hspace{1.5cm}
\begin{tikzpicture}
\node[dot] (v0) at (0,0) {};
\node[dot] (v1) at (-1,0.5) {};
\node[dot] (v2) at (0,0.5) {};
\node[dot] (v3) at (1,0.5) {};
\node[dot] (v4) at (-1,1) {};
\node[dot] (v5) at (0,1) {};
\node[dot] (v6) at (1,1) {};
\path[draw,thick] (v0)--(v1) (v0)--(v2) (v0)--(v3)
(v1)--(v4) (v1)--(v5) (v2)--(v4) (v2)--(v6) (v3)--(v5) (v3)--(v6);
\end{tikzpicture}
\end{center}
One can show that $Q(\delta(\mathbb{P}))$ generates a cover of $\var{Pa}_2$.

\begin{question}
Are there any other covers of $\var{Pa}_2$? More generally, what are the covers of\/
$\var{Pa}_m$? 
\end{question}

\bibliographystyle{plain}
\bibliography{quasivarieties}

\begin{thebibliography}{10}

\bibitem{Ada76}
M.~E. Adams.
\newblock Implicational classes of pseudocomplemented distributive lattices.
\newblock {\em Acta Sci. Math. (Szeged)}, 13(13):381--384, 1976.

\bibitem{AD94}
M.~E. Adams and W.~Dziobiak.
\newblock $q$-universal quasivarieties of algebras.
\newblock {\em Proc. Amer. Math. Soc.}, 120(4):1053--1059, 1994.

\bibitem{AUg24}
Paolo Aglian\`o and Sara Ugolini.
\newblock Structural and universal completeness in algebra and logic.
\newblock {\em Ann. Pure Appl. Logic}, 175(3):Paper No. 103391, 49, 2024.

\bibitem{AC24}
Paolo Aglianó and Alex Citkin.
\newblock Structural completeness in quasivarieties, 2024.

\bibitem{BG71}
R.~Balbes and G.~Gr\"{a}tzer.
\newblock Injective and projective {S}tone algebras.
\newblock {\em Duke Math. J.}, 38:339--347, 1971.

\bibitem{BH70}
R.~Balbes and A.~Horn.
\newblock Stone lattices.
\newblock {\em Duke Math. J.}, 37:537--545, 1970.

\bibitem{Ber91}
C.~Bergman.
\newblock Structural completeness in algebra and logic.
\newblock In {\em Algebraic logic ({B}udapest, 1988)}, volume~54 of {\em
  Colloq. Math. Soc. J\'{a}nos Bolyai}, pages 59--73. North-Holland, Amsterdam,
  1991.

\bibitem{Ber11}
C.~Bergman.
\newblock {\em Universal algebra}, volume 301 of {\em Pure and Applied
  Mathematics (Boca Raton)}.
\newblock CRC Press, Boca Raton, FL, 2012.

\bibitem{BM23}
N.~Bezhanishvili and T.~Moraschini.
\newblock Hereditarily structurally complete intermediate logics: {C}itkin's
  theorem via duality.
\newblock {\em Studia Logica}, 111(2):147--186, 2023.

\bibitem{CM09}
P.~Cintula and G.~Metcalfe.
\newblock Structural completeness in fuzzy logics.
\newblock {\em Notre Dame J. Form. Log.}, 50(2):153--182, 2009.

\bibitem{Cit16}
Alex Citkin.
\newblock Algebraic logic perspective on {P}rucnal's substitution.
\newblock {\em Notre Dame J. Form. Log.}, 57(4):503--521, 2016.

\bibitem{Cit20}
Alex Citkin.
\newblock Hereditarily structurally complete positive logics.
\newblock {\em Rev. Symb. Log.}, 13(3):483--502, 2020.

\bibitem{Doy70}
J.~Doyen.
\newblock Sur la croissance du nombre de syst\`emes triples de steiner non
  isomorphes.
\newblock {\em J. Combinatorial Theory}, 8:424--441, 1970.

\bibitem{DS16}
Wojciech Dzik and Micha\l~M. Stronkowski.
\newblock Almost structural completeness; an algebraic approach.
\newblock {\em Ann. Pure Appl. Logic}, 167(7):525--556, 2016.

\bibitem{Dzi85}
W.~Dziobiak.
\newblock On subquasivariety lattices of some varieties related with
  distributive p-algebras.
\newblock {\em Algebra Universalis}, 21:62--67, 1985.

\bibitem{GLQ80}
G.~Gr\"{a}tzer, H.~Lakser, and R.~W. Quackenbush.
\newblock On the lattice of quasivarieties of distributive lattices with
  pseudocomplementation.
\newblock {\em Acta Sci. Math. (Szeged)}, 42(3-4):257--263, 1980.

\bibitem{Hod93}
W.~Hodges.
\newblock {\em Model theory}, volume~42 of {\em Encyclopedia of Mathematics and
  its Applications}.
\newblock Cambridge University Press, 1993.

\bibitem{HN18}
J.~Hyndman and J.~B. Nation.
\newblock {\em The lattice of subquasivarieties of a locally finite
  quasivariety}.
\newblock CMS Books in Mathematics/Ouvrages de Math\'{e}matiques de la SMC.
  Springer, Cham, 2018.

\bibitem{KS24}
Tomasz Kowalski and Katarzyna Słomczyńska.
\newblock Free p-algebras revisited: an algebraic investigation of
  implication-free intuitionism, 2024.

\bibitem{Lee70}
K.~B. Lee.
\newblock Equational classes of distributive pseudo-complemented lattices.
\newblock {\em Canadian Journal of Mathematics}, 22(4):881–891, 1970.

\bibitem{McK72}
R.~McKenzie.
\newblock Equational bases and nonmodular lattice varieties.
\newblock {\em Trans. Amer. Math. Soc.}, 174:1--43, 1972.

\bibitem{Min72}
G.~E. Minc.
\newblock Derivability of admissible rules.
\newblock {\em Zap. Nau\v{c}n. Sem. Leningrad. Otdel. Mat. Inst. Steklov.
  (LOMI)}, 32:85--89, 156, 1972.

\bibitem{ORvA}
J.~S. Olson, J.~G. Raftery, and C.~J. van Alten.
\newblock Structural completeness in substructural logics.
\newblock {\em Log. J. IGPL}, 16(5):455--495, 2008.

\bibitem{Pog71}
W.~A. Pogorzelski.
\newblock Structural completeness of the propositional calculus.
\newblock {\em Bull. Acad. Polon. Sci. S\'{e}r. Sci. Math. Astronom. Phys.},
  19:349--351, 1971.

\bibitem{PW08}
W.~A. Pogorzelski and P.~Wojtylak.
\newblock {\em Completeness theory for propositional logics}.
\newblock Studies in Universal Logic. Birkh\"{a}user Verlag, Basel, 2008.

\bibitem{Pri75}
H.~A. Priestley.
\newblock The construction of spaces dual to pseudocomplemented distributive
  lattices.
\newblock {\em Quart. J. Math. Oxford Ser. (2)}, 26(102):215--228, 1975.

\bibitem{Qua76}
R.~W. Quackenbush.
\newblock Varieties of {S}teiner loops and {S}teiner quasigroups.
\newblock {\em Canadian J. Math.}, 28(6):1187--1198, 1976.

\bibitem{RS16}
J.~G. Raftery and K.~\'{S}wirydowicz.
\newblock Structural completeness in relevance logics.
\newblock {\em Studia Logica}, 104(3):381--387, 2016.

\bibitem{Wro76}
A.~Wro\'{n}ski.
\newblock The number of quasivarieties of distributive lattices with
  pseudocomplementation.
\newblock {\em Rep. Math. Logic}, (6):111--115, 1976.

\bibitem{Wro81}
A.~Wro\'{n}ski.
\newblock Splittings of lattices of quasivarieties.
\newblock {\em Bulletin of the Section of Logic}, 10(3):128--129, 1981.

\bibitem{Wro09}
A.~Wro\'{n}ski.
\newblock Overflow rules and a weakening of structural completeness.
\newblock In J.~Sytnik-Czetwertyński, editor, {\em Rozważania o filozofii
  prawdziwej. Jerzemu Perzanowskiemu w darze}. Wydawnictwo Uniwersytetu
  Jagiellońskiego, 2009.

\end{thebibliography}

\end{document}